\documentclass[10pt]{article}
\usepackage{lmodern}
\usepackage{amsmath}
\usepackage[T1]{fontenc}
\usepackage[utf8]{inputenc}
\usepackage{authblk}
\usepackage{amsfonts}
\usepackage{graphicx}
\usepackage{color,tikz}

\usepackage{rotating}
\usepackage{amssymb}
\usepackage[english]{babel}
\usepackage{color}
\usepackage{amsthm}
\usepackage{graphicx}
\usepackage{mathrsfs}
\usepackage{makecell}
\usepackage{microtype}
\usepackage{enumerate}
\usepackage{mathscinet}
\usepackage{array}
\usepackage{multirow}
\usepackage{booktabs}
\usepackage{enumerate}
\usepackage[cal=boondoxo,bb=ams]{mathalfa}
\usepackage{hyperref}
\hypersetup{hidelinks}
\usepackage{titlesec}
\newtheorem{defn}{Definition}[section]

\newtheorem{theorem}{Theorem}[section]
\newtheorem{prop}{Proposition}[section]
\newtheorem{lemma}{Lemma}[section]

\newtheorem{remark}{Remark}[section]

\newcommand{\ml}{\mathcal}
\newcommand{\mb}{\mathbb}

\title{A competition on blow-up for semilinear wave equations with scale-invariant damping and nonlinear memory term}
\author{Wenhui Chen\thanks{Wenhui Chen (wenhui.chen.math@gmail.com)}}
\author[2,3]{Ahmad Z. Fino\thanks{Ahmad Z. Fino (ahmad.fino01@gmail.com; a.fino@squ.edu.om)}}
\affil[1]{School of Mathematical Sciences, Shanghai Jiao Tong University, Shanghai, China}
\affil[2]{Department of Mathematics, College of Science, Sultan Qaboos University, Muscat, Oman}
\affil[3]{Department of Mathematics, Faculty of Sciences, Lebanese University, Lebanon}

\date{}
\begin{document}

\maketitle
\begin{abstract}
	\medskip
In this paper, we investigate blow-up of solutions to semilinear wave equations with scale-invariant damping and nonlinear memory term in $\mathbb{R}^n$, which can be represented by the Riemann-Liouville fractional integral of order $1-\gamma$ with $\gamma\in(0,1)$. Our main interest is to study mixed influence from damping term and the  memory kernel on blow-up conditions for the power of nonlinearity, by using test function method or generalized Kato's type lemma. We find a new competition, particularly for the small value of $\gamma$, on the blow-up range between the effective case and the non-effective case.\\
	
	\noindent\textbf{Keywords:} semilinear wave equation, scale-invariant damping, power nonlinearity, nonlinear memory,  Riemann-Liouville fractional integral, blow-up.\\
	
	\noindent\textbf{AMS Classification (2020)} 35B44, 35L05, 35L71, 26A33, 35B33
\end{abstract}
\fontsize{12}{15}
\selectfont
\section{Introduction}
\subsection{Background of wave equations with scale-invariant damping}

In the last two decades, the wave equations with scale-invariant damping have caught a lot of attention. Let us begin with the corresponding linear Cauchy problem with vanishing right-hand side as follows:
\begin{align}\label{Eq_Linear_Wave_Scale}
	\begin{cases}
		\displaystyle{u_{tt}-\Delta u+\frac{\mu}{1+t}u_t=0,}&x\in\mb{R}^n,\ t>0,\\
		(u,u_t)(0,x)=(u_0,u_1)(x),&x\in\mb{R}^n,
	\end{cases}
\end{align}
where $\mu\in(0,\infty)$. According to the classification introduced by \cite{Wirth2004}, thanks to the hyperbolic scaling $
\tilde{u}(t,x):=u(\sigma(t+1)-1,\sigma x)
$
with $\sigma\in(0,\infty)$, the unknown $\tilde{u}=\tilde{u}(t,x)$ satisfies the same wave equation with the damping term, which is the so-called \emph{scale-invariant}. Indeed, the behavior of the solutions to \eqref{Eq_Linear_Wave_Scale} is mainly determined by the parameter $\mu$ in the damping mechanism, which provides a threshold between the \emph{effective damping} and the \emph{non-effective damping}. Here, the effective damping stands for its solution somehow having the behavior of the corresponding parabolic equation, and the non-effective damping means its solution somehow having the behavior of the free wave equation. 

We now turn to the semilinear Cauchy problem with the power source nonlinear term
\begin{align}\label{Eq_Semi_Wave_Scale_Power}
	\begin{cases}
		\displaystyle{u_{tt}-\Delta u+\frac{\mu}{1+t}u_t=|u|^p,}&x\in\mb{R}^n,\ t>0,\\
		(u,u_t)(0,x)=(u_0,u_1)(x),&x\in\mb{R}^n.
	\end{cases}
\end{align}
To introduce the previous researches logically on \eqref{Eq_Semi_Wave_Scale_Power}, we will separate our review according to the value of $\mu$: the \emph{parabolic-like model} for large value of $\mu$ and the \emph{hyperbolic-like model} for small value of $\mu$. To be specific, the parabolic-like model represents the obtained results are related to those for the semilinear heat model, and the hyperbolic-like model stands for the derived results are related to those for the semilinear wave model.
\begin{itemize}
	\item Employing test function method (see \cite{Zhang2001}), the author of \cite{Wakasugi2014} proved blow-up of solutions providing that $1<p\leqslant p_{\mathrm{Fuj}}(n+\mu-1)$ if $\mu\in(0,1]$, and $1<p\leqslant p_{\mathrm{Fuj}}(n)$ if $\mu\in(1,\infty)$,
	in which 
	\begin{align*}
		p_{\mathrm{Fuj}}(n):=1+\frac{2}{n},
	\end{align*}
	is the so-called \emph{Fujita exponent}. The Fujita exponent is the critical exponent for semilinear heat equations. Later, \cite{DAbbicco2015} derived the global existence result providing that $p_{\mathrm{Fuj}}(n)<p\leqslant \frac{n}{(n-2)_+}$, with the value of $\mu$ such that $\mu\geqslant \frac{5}{3}$ for $n=1$, $\mu\geqslant 3$ for $n=2$, and $\mu\geqslant n+2$ for $n\geqslant 3$. Therefore, these results show the critical exponent for \eqref{Eq_Semi_Wave_Scale_Power} is the Fujita exponent with a large parameter $\mu$. 
	\item Let us consider the ``not large'' value of $\mu$. In the special case $\mu=2$, the authors of \cite{DAbbicco-Lucente-Reissig2015} proved that for $n=1,2,3,$ the critical exponent is given by the competition $\max\{p_{\mathrm{Fuj}}(n),p_{\mathrm{Str}}(n+2)\}$, where the so-called \emph{Strauss exponent} $p_{\mathrm{Str}}(n)$ is the critical exponent for semilinear wave equations. To be specific, the Strauss exponent is written by
	\begin{align*}
		p_{\mathrm{Str}}(n):=\frac{n+1+\sqrt{n^2+10n-7}}{2(n-1)},
	\end{align*}
	and it is the positive root of the quadratic equation $(n-1)p^2-(n+1)p-2=0$. In the same year, global existence results were extended for some $p>p_{\mathrm{Str}}(n)$ in odd dimensions $n\geqslant 5$ in  \cite{DAbbicco-Lucente2015}. The global existence results for general dimensional cases were derived in \cite{Palmieri2019,Palmieri201902}. Recently, \cite{Lai-Takamura-Wakasa2017} found a shifted Strauss exponent for the blow-up results, to be specific, $p_{\mathrm{Fuj}}(n)\leqslant p<p_{\mathrm{Str}}(n+2\mu)$ when $\mu\in(0,\frac{\mu^*}{2})$, where $\mu^*=\frac{n^2+n+2}{n+2}$. By the aid of hypergeometric functions motivated by \cite{Zhou-Han2014}, the authors of \cite{Ikeda-Sobajima2018} got the sharper blow-up results if $p_{\mathrm{Fuj}}(n)\leqslant p\leqslant p_{\mathrm{Str}}(n+\mu)$ when $\mu\in(0,\mu^*)$.
	It was also conjectured that the critical exponent could be $p_{\mathrm{Str}}(n+\mu)$ for some small value of $\mu$. Under $1<p\leqslant p_{\mathrm{Str}}(n+\mu)$ with $\mu\in(0,\infty)$, the lifespan estimates are improved in the papers \cite{Tu-Lin201701,Tu-Lin201702,Palmieri-Tu2019} by applying iteration argument associated with modified Bessel functions. Concerning other studies on semilinear scale-invariant damped wave equations, we refer to \cite{Lai2018,Kato-Takamura-Wakasa2019,Imai-Kato-Takamura-Wakasa2019,Lai-Schiavone-Takamura2020,Hamouda-Hamaz=2021,Lai-Zhou=2021,Kitamura-Morisawa-Takamura=2022}.
\end{itemize}

\subsection{Background of wave equations with nonlinear memory term}
We now recall some studies for semilinear damped wave equations with nonlinear memory term. Let us start by the case of constant coefficient as follows:
\begin{align}\label{Eq_Semi_Damped_Constant_Wave_Memory}
	\begin{cases}
		u_{tt}-\Delta u+u_t=N_{\gamma,p}[u],&x\in\mb{R}^n,\ t>0,\\
		(u,u_t)(0,x)=(u_0,u_1)(x),&x\in\mb{R}^n,
	\end{cases}
\end{align}
where the nonlinear term on the right-hand side of the equation is the Riemann-Liouville fractional integral of order $1-\gamma$ with the $p$ power of the solution, i.e.
\begin{align}\label{memory nonlinearity}
	N_{\gamma,p}[u]:=c_{\gamma}\int_0^t(t-\tau)^{-\gamma}|u(\tau, x)|^p\mathrm{d}\tau \ \ \mbox{carrying} \ \ c_{\gamma}:=\frac{1}{\Gamma(1-\gamma)},
\end{align}
where $p>1$, $\gamma\in(0,1)$ and $\Gamma$ denotes the Euler integral of the second kind. According to the recent studies of \cite{Fino2011,Berbiche2013,DAbb14,Berbiche2015}, the critical exponent to \eqref{Eq_Semi_Damped_Constant_Wave_Memory} is given by the competition such that $\max\{p_{\gamma}(n),\frac{1}{\gamma}\}$, where we wrote
\begin{align*}
	p_{\gamma}(n):=1+\frac{2(2-\gamma)}{(n-2(1-\gamma))_+}.
\end{align*}
Later, the authors of \cite{Dannawi-Kirane-Fino2018} investigated blow-up of solutions to
\begin{align}\label{Eq_Semi_Damped_Variable_Wave_Memory}
	\begin{cases}
		u_{tt}-\Delta u+a(x)b(t)u_t=N_{\gamma,p}[u],&x\in\mb{R}^n,\ t>0,\\
		(u,u_t)(0,x)=(u_0,u_1)(x),&x\in\mb{R}^n,
	\end{cases}
\end{align}
where the time-space-dependent coefficient in the damping term is defined by
\begin{align}\label{For_Time_Space}
	a(x)b(t):=a_0(1+|x|^2)^{-\frac{\alpha}{2}}(1+t)^{-\beta},
\end{align}
with $a_0>0$, $\alpha,\beta\geqslant0$ and $\alpha+\beta<1$. Roughly speaking, \cite{Dannawi-Kirane-Fino2018} just considered the effective case due to $\alpha+\beta<1$. Furthermore, some blow-up and global existence results for $\alpha=0$ and $\beta\in(-1,1)$ were obtained in the recent papers \cite{Haj-Reissig-2020} and \cite{Haj-Reissig=2021}, respectively. To the best of the authors' knowledge, so far it is still unknown for the existence/nonexistence of global (in time) solutions for the scale-invariant case of time-dependent type, i.e. \eqref{For_Time_Space} with $a_0=\mu$, $\alpha=0$ and $\beta=1$. In this paper, we will give an answer of blow-up  of solutions, where we may observe a competition between the behavior for the parabolic-like model and the hyperbolic-like model influenced by the relaxation function in the nonlinear memory term.

Finally, we recall the recent result for the semilinear wave equation with nonlinear memory term
\begin{align}\label{Eq_Semi_Wave_Memory}
	\begin{cases}
		\displaystyle{u_{tt}-\Delta u=N_{\gamma,p}[u],}&x\in\mb{R}^n,\ t>0,\\
		(u,u_t)(0,x)=(u_0,u_1)(x),&x\in\mb{R}^n.
	\end{cases}
\end{align}
The authors of \cite{Chen-Palmieri201903} proved blow-up of energy solutions to \eqref{Eq_Semi_Wave_Memory} if $p>1$ for $n=1$ and $1<p\leqslant p_0(n,\gamma)$ for $n\geqslant 2$, where $p=p_0(n,\gamma)$ solves $(n-1) p^2-\left(n+3-2\gamma\right)p -2=0$, with $\gamma\in(0,1)$ and $p>1$. Here, for any $n\geqslant 2$ we denote by $p_0(n,\gamma)$ the positive root of the last  quadratic equation
\begin{align}\label{p0(n,g)}
	p_0(n,\gamma):=\frac{n+3-2\gamma+\sqrt{n^2+(14-4\gamma)n+(3-2\gamma)^2-8}}{2(n-1)}.
\end{align} In the case $n=1$ we set formally  $p_0(1,\gamma)= \infty$. This is a generalized Strauss exponent and satisfies $\lim_{\gamma\to 1^-}p_0(n,\gamma)=p_{\mathrm{Str}}(n)$ for all $n\geqslant2$. Furthermore, the research concerning blow-up of solutions to \eqref{Eq_Semi_Wave_Memory} with general nonlinear memory terms $g\ast|u|^p$ has been done recently in \cite{Chen-2020-04}.

\subsection{Main purpose of this paper}
In the present work, we study blow-up of solutions to the Cauchy problem for semilinear wave equations with scale-invariant damping of time-dependent type as well as nonlinear memory term, namely,
\begin{align}\label{Eq_Semi_Wave_Scale_Memory}
	\begin{cases}
		\displaystyle{u_{tt}-\Delta u+\frac{\mu}{1+t}u_t=N_{\gamma,p}[u],}&x\in\mb{R}^n,\ t>0,\\
		(u,u_t)(0,x)=(u_0,u_1)(x),&x\in\mb{R}^n,
	\end{cases}
\end{align}
with $\mu\in(0,\infty)$ and $\gamma\in(0,1)$, where the nonlinearity has been defined by \eqref{memory nonlinearity}. Especially, we are interested in the influence from various kinds of damping term (dominant by parameter $\mu$) and different behaviors of nonlinear memory kernel (dominant by parameter $\gamma$) on blow-up conditions describing by the upper bounds of the exponent $p$. 

Our novelty is a new competition of the hyperbolic-like model and the parabolic-like model. For example, even for small $\mu$ (i.e. the hyperbolic-like classification to the classical model \eqref{Eq_Semi_Wave_Scale_Power}), we still may understand our aim model \eqref{Eq_Semi_Wave_Scale_Memory} as the parabolic-like one when the parameter $\gamma$ is small. That is to say that the memory effect in the nonlinearity will exert crucial influence on the behaviors for the semilinear Cauchy problem. To achieve our goal, we will apply suitable test function methods associated with a flexible time-scaling parameter to obtain blow-up result under the parabolic-like situation. In addition, after proving generalized Kato's type lemma, we demonstrate blow-up result under the hyperbolic-like situation by employing iteration methods.

\medskip

\noindent\textbf{Notation:}\ \  We give some notations to be used in this paper. $f\lesssim g$ means that there exists a positive constant $C$ such that $f\leqslant Cg$.
$B_R$ denotes the ball around the origin with radius $R$ in $\mathbb{R}^n$. Moreover, $(a)_+:=\max\{a,0\}$ stands for the positive part of the real number.

\section{Main results}
\subsection{Blow-up result for the parabolic-like model}
According to the recent paper \cite{Dannawi-Kirane-Fino2018}, one may derive local (in time)  well-posedness for \eqref{Eq_Semi_Wave_Scale_Memory} in the following sense.
\begin{lemma}\label{Lem_DKF}
	Let $n\geqslant 1$, $\gamma\in(0,1)$ and $1<p\leqslant \frac{n}{(n-2)_+}$. Let us consider initial data $u_0\in H^1(\mb{R}^n)$ and $u_1\in L^2(\mb{R}^n)$. Then, there is a unique maximal mild solution $u=u(t,x)$ to the Cauchy problem \eqref{Eq_Semi_Wave_Scale_Memory} such that
	\begin{align*}
		u\in\ml{C}\big([0,T),H^1(\mb{R}^n)\big)\cap \ml{C}^1\big([0,T),L^2(\mb{R}^n)\big),
	\end{align*}
	where $0<T\leqslant\infty$. Particularly, we say $u$ is a global (in time) solution to \eqref{Eq_Semi_Wave_Scale_Memory} if $T=\infty$, while in the case $T<\infty$, we say $u$ blows up in finite time, namely 
	\begin{align*}
		\|u(t,\cdot)\|_{H^1}+\|u_t(t,\cdot)\|_{L^2}\rightarrow\infty\quad\mbox{as}\quad
		t\rightarrow T_{\max}.
	\end{align*}
Additionally, if we assume $\mathrm{supp}\,u_0,u_1\subset B_R$, then we claim that $\mathrm{supp}\,u(t,\cdot)\subset B_{R+t}$ for some $R>0$.
\end{lemma}

In order to describe our first blow-up result, we introduce 
\begin{align*}
	p_1(n):=1+\frac{3-\gamma}{n-1+\gamma}\ \ &\mbox{and}\ \ p_2(n,d_0):=1+\frac{2-\gamma}{(nd_0-2+\gamma)_+}\ \ \mbox{if}\ \ \mu\in(1,\infty),\\
	p_3(n):=1+\frac{3-\gamma}{(n+\mu+\gamma-2)_+}\ \ &\mbox{and}\ \ p_4(n,d_1):=1+\frac{2-\gamma}{(nd_1-2+\gamma)_+}\ \ \mbox{if}\ \ \mu\in(0,1],
\end{align*}
carrying the parameters $d_k=d_k(n)$ with $k=0,1$ such that
\begin{align*}
d_0(n):=\frac{1}{4}+\sqrt{\frac{1}{16}+\frac{2-\gamma}{n}}\ \ \mbox{as well as}\ \ d_1(n):=\frac{1}{4}+\sqrt{\frac{1}{16}+\frac{(\mu+1)(2-\gamma)}{2n}}.
\end{align*}
\begin{remark}\label{Rem_Fujita}
	Indeed, we observe two limits that
	\begin{align*}
		\lim\limits_{\gamma\to 1^-}p_1(n)=p_{\mathrm{Fuj}}(n)\ \ \mbox{and}\ \ 	\lim\limits_{\gamma\to 1^-}p_3(n)=p_{\mathrm{Fuj}}(n+\mu-1).
	\end{align*}
	By taking the consideration of
	\begin{align*}
		\lim_{\gamma\to 1^{-}} c_\gamma s_+^{-\gamma} = \delta_0(s) \ \ \mbox{in the sense of distributions with} \ \  s_+^{-\gamma}:= \begin{cases} s^{-\gamma} & \mbox{if} \ \ s>0, \\ 0 & \mbox{if} \ \ s<0, \end{cases}
	\end{align*}
with the formal motivation of the nonlinear term
\begin{align*}
 \lim\limits_{\gamma\to 1^{-}}N_{\gamma,p}[u]= \lim\limits_{\gamma\to 1^{-}}c_{\gamma}t^{-\gamma}\ast|u|^p=|u|^p,
\end{align*}
 two exponents $p_1(n)$, $p_3(n)$ in the above are the natural extensions of those in \cite{Wakasugi2014} when $\mu\in(1,\infty)$ and $\mu\in(0,1]$, respectively.
\end{remark}
\begin{theorem}\label{Thm_Blow_Up_mu_1}
	Let us assume that
	\begin{align*}
		1<p\leqslant\begin{cases}
		p_1(n)&\mbox{if}\ \ n=1,2,\\
		\min\{p_1(3),p_2(3,d_0)\}&\mbox{if}\ \ n=3,\\
		\min\{p_2(n,d_0),\frac{n}{n-2}\}&\mbox{if}\ \ n\geqslant4,
		\end{cases}
	\end{align*}
when $\mu\in(1,\infty)\setminus\{2\}$; 
$$1<p\leqslant	\min\left\{p_1(n),\frac{n}{(n-2)_+}\right\},$$
when $\mu=2$; and
\begin{align*}
1<p\leqslant\begin{cases}
p_3(1)&\mbox{if}\ \ n=1,\\
\min\{p_3(2),p_4(2,d_1) \}&\mbox{if}\ \ n=2,\\
\min\{p_3(3),p_4(3,d_1), 3\}&\mbox{if}\ \ n=3,\\
\min\{p_4(n,d_1),\frac{n}{n-2}\}&\mbox{if}\ \ n\geqslant4,
\end{cases}
\end{align*}	
when $\mu\in(0,1]$;	for all $n\geqslant 1$ and $\gamma\in(0,1)$. Let us consider initial data $u_0\in H^1(\mb{R}^n)\cap L^1(\mb{R}^n)$ and $u_1\in L^2(\mb{R}^n)\cap L^1(\mb{R}^n)$ satisfying 
\begin{align*}
\int_{\mb{R}^n}\big(u_1(x)-(\mu-1)u_0(x)\big)\mathrm{d}x>0\ \ &\mbox{if} \ \ \mu\in(1,\infty),\\
 \int_{\mb{R}^n}u_1(x)\mathrm{d}x>0\ \ &\mbox{if}\ \ \mu\in(0,1].
\end{align*}
Then, the mild solution to the Cauchy problem \eqref{Eq_Semi_Wave_Scale_Memory} blows up in finite time.
\end{theorem}
\subsection{Blow-up result for the hyperbolic-like model}

Next, we will show the blow-up result when $1<p<p_0(n+\mu,\gamma)$ for any $\mu\in(0,\infty)$ and $\gamma\in(0,1)$, in which $p_0(n,\gamma)$ was defined in \eqref{p0(n,g)}. It seems reasonable to derive in blow-up results an upper bound $p_0(n+\mu,\gamma)$ if $\mu$ is small since this exponent fulfills
\begin{align*}
	\lim\limits_{\gamma\to 1^-} p_0(n+\mu,\gamma)=p_{\mathrm{Str}}(n+\mu),
\end{align*}
where $p_{\mathrm{Str}}(n+\mu)$ is the conjectured critical exponent for the classical model \eqref{Eq_Semi_Wave_Scale_Power} carrying small $\mu$. A similar explanation was shown in Remark \ref{Rem_Fujita}.

Before stating our second result, let us introduce the weak solution to the Cauchy problem \eqref{Eq_Semi_Wave_Scale_Memory}.
\begin{defn}\label{Defn.Energy.Solution_2}
	Let  $u_0\in L^1_{\mathrm{loc}}(\mb{R}^n)$ and $u_1\in L^1_{\mathrm{loc}}(\mb{R}^n)$. We say that $u$ is a weak solution of \eqref{Eq_Semi_Wave_Scale_Memory} on $[0,T)$ if
	\begin{align*}
		u,\, u_t,\,  N_{\gamma,p}[u]\in L^1_{\mathrm{loc}}\big([0,T)\times\mathbb{R}^n\big)
	\end{align*}  
	and the integral relation
	\begin{align}\label{Eq.Defn.Energy.Solution_2}
		&\int_{\mb{R}^n}u_t(t,x) \psi(t,x)\mathrm{d}x- \int_{\mb{R}^n}u_1(x) \psi(0,x)\mathrm{d}x -\int_0^t\int_{\mb{R}^n}\big(u(s,x)\Delta\psi(s,x)+u_t(s,x) \psi_s(s,x)\big)\mathrm{d}x\mathrm{d}s\notag\\
		&+\int_0^t\int_{\mb{R}^n}\frac{\mu u_t(s,x)}{1+s}\psi(s,x)\mathrm{d}x\mathrm{d}s=c_\gamma\int_0^t\int_{\mb{R}^n}\psi(s,x)\int_0^s(s-\tau)^{-\gamma}|u(\tau,x)|^p\mathrm{d}\tau\mathrm{d}x\mathrm{d}s,
	\end{align}
	holds for any $\psi\in\ml{C}_c^{2}\big([0,T)\times\mb{R}^n\big)$ and any $t\in[0,T)$.
\end{defn}

  Let us begin with stating the second blow-up result to \eqref{Eq_Semi_Wave_Scale_Memory}.

\begin{theorem}\label{Thm_Blow_Up_mu_0}
	Let us assume $p\leqslant \frac{n}{(n-2)_+}$ and
	\begin{align}\label{Condition_mu_0}
		1<p< p_{0}(n+\mu,\gamma)
	\end{align}
	for all $n\geqslant 1$ and $\gamma\in(0,1)$. Let  $u_0\in H^1(\mb{R}^n)$ and $u_1\in L^2(\mb{R}^n)$ be nonnegative, nontrivial and compactly supported functions with supports contained in $B_R$ for some $R\geqslant 1$. Then, the mild solution to the Cauchy problem \eqref{Eq_Semi_Wave_Scale_Memory} blows up in finite time.
\end{theorem}
\begin{remark}
	The phenomenon of shifted Strauss type exponent (see, for example, \cite{DAbbicco-Lucente-Reissig2015,Ikeda-Sobajima2018,Tu-Lin201701,Palmieri-Tu2019}) is exactly the same as those for the power nonlinearity \eqref{Eq_Semi_Wave_Scale_Power}. Precisely, concerning the subcritical case, the recent paper \cite{Chen-Palmieri201903} proved blow-up of solutions to the semilinear wave equation with memeory nonlinearity \eqref{Eq_Semi_Wave_Memory} if $1<p<p_0(n,\gamma)$ for $\gamma\in(0,1)$. In Theorem \ref{Thm_Blow_Up_mu_0}, the scale-invariant behavior can be expressed by a shifted of Strauss type exponent $p_0(n,\gamma)$ to $p_0(n+\mu,\gamma)$.
\end{remark}
\begin{remark}
	Let us consider the Cauchy problem \eqref{Eq_Semi_Wave_Scale_Memory} with $\mu=2$. By performing the Liouville transform $v(t,x):=(1+t)u(t,x)$, which satisfies
	\begin{align}\label{Eq_Wave}
		v_{tt}-\Delta v=c_{\gamma}(1+t)\int_0^t(t-\tau)^{-\gamma}(1+\tau)^{-p}|v(\tau,x)|^p\mathrm{d}\tau,
	\end{align}
	we also can prove blow-up of mild solution in the critical case $p=p_0(n+2,\gamma)$. Its proof is similar to those in \cite{Chen-Palmieri201903}. Furthermore, the main technique is an iteration method associated with auxiliary functions (proposed in \cite{Ikeda-Sobajima2018}) since the benefit of taking nonlinear wave equation \eqref{Eq_Wave} with time-dependent part only. Nevertheless, the blow-up result in the other critical cases $p=p_0(n+\mu,\gamma)$ if $\mu\neq 2$ is still open.
\end{remark}

\subsection{Some explanations for the competition: hyperbolic-like versus parabolic-like}
Let us summarize the derived results in Theorems \ref{Thm_Blow_Up_mu_1} and \ref{Thm_Blow_Up_mu_0}. By taking some conditions for initial datum, we may claim blow-up of solutions to the Cauchy problem \eqref{Eq_Semi_Wave_Scale_Memory} provided the following restrictions for the power exponent hold:
	\begin{align*}
	1<p\leqslant\begin{cases}
		\max\big\{p_1(n),p^-_0(n+\mu,\gamma)\big\}&\mbox{if}\ \ n=1,2,\\
		\max\big\{\min\{p_1(3),p_2(3,d_0)\},p^-_0(3+\mu,\gamma)\big\}&\mbox{if}\ \ n=3,\\
		\max\big\{\min\{p_2(n,d_0),\frac{n}{n-2}\},p^-_0(n+\mu,\gamma)\big\}&\mbox{if}\ \ n\geqslant4,
	\end{cases}
\end{align*}
when $\mu\in(1,\infty)\setminus\{2\}$; 
$$1<p\leq \max\big\{p_1(n),p^-_0(n+\mu,\gamma)\big\},\quad\mbox{for all}\ \ n\geqslant1,$$
when $\mu=2$; and
\begin{align*}
	1<p\leqslant\begin{cases}
		\max\big\{p_3(1),p^-_0(1+\mu,\gamma)\big\}&\mbox{if}\ \ n=1,\\
		\max\big\{\min\{p_3(2),p_4(2,d_1) \},p^-_0(2+\mu,\gamma)\big\}&\mbox{if}\ \ n=2,\\
		\max\big\{\min\{p_3(3),p_4(3,d_1), 3\},p^-_0(3+\mu,\gamma)\big\}&\mbox{if}\ \ n=3,\\
		\max\big\{\min\{p_4(n,d_1),\frac{n}{n-2}\},p^-_0(n+\mu,\gamma)\big\}&\mbox{if}\ \ n\geqslant4,
	\end{cases}
\end{align*}	
when $\mu\in(0,1]$; for all $\gamma\in(0,1)$, where $p^-_0(n+\mu,\gamma):=p_0(n+\mu,\gamma)-\epsilon_0$ for a sufficiently small constant $\epsilon_0>0$.  We underline that the competition in the above is different from those in semilinear wave equations with scale-invariant damping and power nonlinearity \eqref{Eq_Semi_Wave_Scale_Power}. 

Let us focus on the subcritical case. Concerning the memoryless Cauchy problem \eqref{Eq_Semi_Wave_Scale_Power}, as shown in the introduction, the critical exponent was conjectured by the Strauss exponent $p_{\mathrm{Str}}(n+\mu)$ if $\mu$ is small, and the Fujita exponent $p_{\mathrm{Fuj}}(n)$ if $\mu$ is large. In other words, the wave equation with scale-invariant damping and power nonlinearity is explained by a hyperbolic-like model if $\mu$ is not large, and by a parabolic-like model if $\mu$ is large. Therefore, the competition, which is determined by the value of $\mu$, between hyperbolic-like and parabolic-like appears.

Nevertheless, our results in blow-up of solutions indicate another competition which is strongly relies on the value of $\gamma$ also. A new phenomenon shows parabolic-like model for small $\mu$ instead of hyperbolic-like model, providing that we take $\gamma$ is also a small parameter. To explain this phenomenon by an 	unambiguous manner, we will concentrate on two dimensional case as an example.

Let us take $n=2$ in Theorems \ref{Thm_Blow_Up_mu_1} and \ref{Thm_Blow_Up_mu_0}. Then, we assert blow-up of solution to the Cauchy problem \eqref{Eq_Semi_Wave_Scale_Memory} if
\begin{align*}
1<p<\max\left\{\frac{4}{1+\gamma},\frac{5+\mu-2\gamma}{2(1+\mu)}+\frac{1}{2(1+\mu)}\sqrt{\mu^2+18\mu-4\mu\gamma+4\gamma^2-20\gamma+33}\right\},
\end{align*}
when $\mu\in(1,\infty)$; and
\begin{align}\label{Condition_small}
1<p<\max\left\{\min\left\{ \frac{\mu+3}{\mu+\gamma},\frac{2d_1}{(2d_1-2+\gamma)_+}\right\} ,\frac{5+\mu-2\gamma+\sqrt{\mu^2+18\mu-4\mu\gamma+4\gamma^2-20\gamma+33}}{2(1+\mu)}\right\}
\end{align}
with $d_1(2)=\frac{1}{4}+\frac{1}{2}\sqrt{\frac{1}{4}+(\mu+1)(2-\gamma)}$ when $\mu\in(0,1]$. Indeed, there is a constant $0<\bar{\gamma}_0\ll 1$ such that if $\gamma\in(0,\bar{\gamma}_0)$, the blow-up condition \eqref{Condition_small} will be reduced to
\begin{align*}
	1<p<\frac{\mu+3}{\mu+\gamma}\ \ \mbox{when} \ \ 0<\mu<\bar{\mu}_0\ll 1.
\end{align*}
Namely, the Fujita type exponent $p_3(2)=\frac{\mu+3}{\mu+\gamma}$ plays a crucial role in the blow-up result rather than the Strauss type exponent. This is the new effect that never happens in the memoryless model \eqref{Eq_Semi_Wave_Scale_Power} since $\gamma\to 1^-$ formally in the memoryless situation.

\section{Proof of Theorem \ref{Thm_Blow_Up_mu_1} via test function method}
\subsection{Preliminaries}
In this subsection, we will recall some basic definitions and useful lemmas for treatments of fractional derivatives and integration that will be used later. According to Chapter 1 in \cite{Samko-Kilbas-Marichev-1987}, the Riemann-Liouville fractional integrals and their derivatives can be shown by the next definitions.
\begin{defn}
A function $\mathcal{A} :[a,b]\rightarrow\mathbb{R}$, $-\infty<a<b<\infty$, is said to be absolutely continuous if and only if there exists $\psi\in L^1(a,b)$ such that
\[
\mathcal{A} (t)=\mathcal{A} (a)+\int_{a}^t \psi(s)\,ds.
\]
$AC[a,b]$ denotes the space of these functions.  Moreover, 
\[
AC^{k}[a,b]:=\left\{\varphi:[a,b]\rightarrow\mathbb{R}\;:\;\varphi^{(k)}\in
AC[a,b],\,\mbox{for all}\,\,k\in\mathbb{N}\right\}.
\]
\end{defn}

\begin{defn}
	Let $f\in L^1(0,T)$ with $T>0$. The Riemann-Liouville left- and right-sides fractional integrals of order $\alpha\in(0,1)$ are
	\begin{align*}
		I_{0|t}^{\alpha}f(t)&:=\frac{1}{\Gamma(\alpha)}\int_0^t(t-s)^{-(1-\alpha)}f(s)\mathrm{d}s\ \ \,\, \mbox{for}\ \ t>0,\\
		I_{t|T}^{\alpha}f(t)&:=\frac{1}{\Gamma(\alpha)}\int_t^T(s-t)^{-(1-\alpha)}f(s)\mathrm{d}s\ \ \mbox{for}\ \ t<T.
	\end{align*}
\end{defn}
\begin{defn}
	Let $f\in \mathrm{AC}[0,T]$ with $T>0$, i.e. $f$ is an absolutely continuous functions. The Riemann-Liouville left- and right-sides fractional derivatives of order $\alpha\in(0,1)$ are
	\begin{align*}
		D_{0|t}^{\alpha}f(t)&:=\frac{\mathrm{d}}{\mathrm{d}t}I_{0|t}^{1-\alpha}f(t)=\frac{1}{\Gamma(1-\alpha)}\frac{\mathrm{d}}{\mathrm{d}t}\int_0^t(t-s)^{-\alpha}f(s)\mathrm{d}s\ \ \ \ \,\, \mbox{for}\ \ t>0,\\
		D_{t|T}^{\alpha}f(t)&:=\frac{\mathrm{d}}{\mathrm{d}t}I_{t|T}^{1-\alpha}f(t)=-\frac{1}{\Gamma(1-\alpha)}\frac{\mathrm{d}}{\mathrm{d}t}\int_t^T(s-t)^{-\alpha}f(s)\mathrm{d}s\ \ \mbox{for}\ \ t<T.
	\end{align*}
\end{defn}

Now, we will show some rules in the calculation of fractional derivatives, which were introduced in the books \cite{Samko-Kilbas-Marichev-1987,Kilbas-Sri-Truj-2006}.
\begin{prop}\label{Prop_3.1}
	Let $T>0$ and $\alpha\in(0,1)$. The fractional integration by parts
	\begin{align}\label{IP}
		\int_{0}^Tf(t)D_{0|t}^{\alpha}g(t)\mathrm{d}t=\int_{0}^Tg(t)D_{t|T}^{\alpha}f(t)\mathrm{d}t
	\end{align}
	holds for every $f\in I_{t|T}^{\alpha}(L^p(0,T))$, $g\in I_{0|t}^{\alpha}(L^q(0,T))$, where $\frac{1}{p}+\frac{1}{q}\leqslant 1+\alpha$ with $p,q>1$ and
	\begin{align*}
		I_{0|t}^{\alpha}\big(L^q(0,T)\big)&:=\left\{f=I_{0|t}^{\alpha}h \ \ \mbox{for}\ \ h\in L^q(0,T)\right\},\\
		I_{t|T}^{\alpha}\big(L^p(0,T)\big)&:=\left\{f=I_{t|T}^{\alpha}h \ \  \mbox{for}\ \ h\in L^p(0,T)\right\}.
	\end{align*}
\end{prop}

\begin{prop}\label{Prop_3.2}
	Let $T>0$ and $\alpha\in(0,1)$. The following identities hold:
	\begin{align*}
		D_{0|t}^{\alpha}I_{0|t}^{\alpha}f(t)&=f(t)\ \ \mbox{a.e.}\ t\in(0,T)\ \ \mbox{for all}\ \ f\in L^p(0,T),\\
		(-1)^kD^k D_{t|T}^{\alpha}f(t)&=D_{t|T}^{k+\alpha}f(t)\ \ \,\,  \qquad\qquad \mbox{for all}\ \ f\in \mathrm{AC}^{k+1}[0,T],
	\end{align*}
	where $1\leqslant p\leqslant \infty$ and $k\in\mb{N}$.
\end{prop}
\begin{remark}
A simple sufficient condition for functions $f$ and $g$ to satisfy \eqref{IP} is that $f,g\in \ml{C}[c,d],$ such that
$D^\alpha_{t|d}f(t),D^\alpha_{c|t}g(t)$ exist at every point $t\in[c,d]$ and are continuous.
\end{remark}
For the sake of clarity, let us consider $w=w(t)$ such that
\begin{align}\label{Defn_w(t)}
	w(t):=\left(1-\frac{t}{T}\right)^{\sigma}\ \ \mbox{for any}\ \ t\in[0,T],
\end{align}
with a large parameter $\sigma\gg 1$. According to \cite[Property 2.1]{Kilbas-Sri-Truj-2006}, its fractional derivatives fulfill
\begin{align}\label{Frac_Cal_01}
	D_{t|T}^{k+\alpha}w(t)=\frac{\Gamma(\sigma+1)}{\Gamma(\sigma+1-k-\alpha)}T^{-(k+\alpha)}\left(1-\frac{t}{T}\right)^{\sigma-(\alpha+k)}\ \ \mbox{for any}\ \ t\in[0,T],
\end{align}
where $T>0$, $\alpha\in(0,1)$ and $k\in\mb{N}$.

\begin{lemma}\label{Lem_Fraction_1}
	Let $T\gg1$, $\alpha\in(0,1)$, $k\in\mb{N}$ and $p>1$. Then, the next estimates hold:
	\begin{align}\label{Est_Lem_Frac_01}
		\int_0^T(1+t)w(t)^{-\frac{1}{p-1}}\left|D_{t|T}^{k+\alpha}w(t)\right|^{\frac{p}{p-1}}\mathrm{d}t&\lesssim T^{2-\frac{(k+\alpha)p}{p-1}},\\
				\int_0^T\big((1+t)w(t)\big)^{-\frac{1}{p-1}}\left|D_{t|T}^{k+\alpha}w(t)\right|^{\frac{p}{p-1}}\mathrm{d}t&\lesssim \ml{D}_p(T) T^{-\frac{(k+\alpha)p}{p-1}},\label{Est_Lem_Frac_02}
	\end{align}
where the time-dependent function is denoted by
\begin{align*}
\ml{D}_p(T):=\begin{cases}
T^{\frac{p-2}{p-1}}&\mbox{if}\ \ p>2,\\
\ln T&\mbox{if}\ \ p=2,\\
1&\mbox{if}\ \ p<2.
\end{cases}
\end{align*}
\end{lemma}
\begin{proof}
	From the definition of $w(t)$, we have
	\begin{align*}
		\mbox{LHS of \eqref{Est_Lem_Frac_01}}\lesssim T^{-\frac{(k+\alpha)p}{p-1}}\int_0^T(1+t)\left(1-\frac{t}{T}\right)^{\sigma-\frac{(k+\alpha)p}{p-1}}\mathrm{d}t.
	\end{align*}
Concerning $\sigma\gg 1$, it holds that
	\begin{align*}
		\int_0^T(1+t)\left(1-\frac{t}{T}\right)^{\sigma-\frac{(k+\alpha)p}{p-1}}\mathrm{d}t\lesssim\int_0^T(1+t)\mathrm{d}t\lesssim T^2
	\end{align*}
	for $T\gg1$, which implies our desired estimates \eqref{Est_Lem_Frac_01}. For another estimate, with the aid of similar approach to the last one and \eqref{Frac_Cal_01}, it gives
	\begin{align*}
		\mbox{LHS of \eqref{Est_Lem_Frac_02}}&\lesssim T^{-\frac{(k+\alpha)p}{p-1}}\int_0^T(1+t)^{-\frac{1}{p-1}}\left(1-\frac{t}{T}\right)^{\sigma-\frac{(k+\alpha)p}{p-1}}\mathrm{d}t\\
		&\lesssim T^{-\frac{(k+\alpha)p}{p-1}}\int_0^T(1+t)^{-\frac{1}{p-1}}\mathrm{d}t.
	\end{align*}
	From the next fact that $		\int_0^T(1+t)^{-\frac{1}{p-1}}\mathrm{d}t\lesssim\ml{D}_p(T)$, the proof is complete.
\end{proof}

\subsection{Blow-up result in the case $\mu\in(1,\infty)$}
Let us now designate $g=g(t)$ by 
\begin{align}\label{function g}
	g(t):=\frac{t+1}{\mu-1}\ \ \mbox{for any}\ \ \mu>1,
\end{align}
so that from the equation in \eqref{Eq_Semi_Wave_Scale_Memory}, we have
\begin{align*}
	gN_{\gamma,p}[u]=(gu)_{tt}-\Delta(gu)-(g'u)_t+u_t
\end{align*}
because of $g'(t)+1=\frac{\mu}{1+t}g(t)$

Furthermore, we assume, on the contrary, that $u=u(t,x)$ is a global (in time) mild solution to \eqref{Eq_Semi_Wave_Scale_Memory}, then as mild solutions being weak solutions (see, for example, \cite{Dannawi-Kirane-Fino2018}) and regarding $g(t)\psi(t,x)$ as a test function, one can derive
\begin{align}\label{Weak_Solution_01}
	&\int_0^T\int_{\mb{R}^n}N_{\gamma,p}[u](t,x)g(t)\psi(t,x)\mathrm{d}x\mathrm{d}t-\int_{\mb{R}^n}u_0(x)g(0)\psi_t(0,x)\mathrm{d}x\notag\\
	&\qquad=-\int_{\mb{R}^n}\big(u_1(x)g(0)+u_0(x)\big)\psi(0,x)\mathrm{d}x+\int_0^T\int_{\mb{R}^n}u(t,x)g(t)\psi_{tt}(t,x)\mathrm{d}x\mathrm{d}t\notag\\
	&\qquad\quad+\int_0^T\int_{\mb{R}^n}u(t,x)\big(g'(t)-1\big)\psi_t(t,x)\mathrm{d}x\mathrm{d}t-\int_0^T\int_{\mb{R}^n}u(t,x)g(t)\Delta\psi(t,x)\mathrm{d}x\mathrm{d}t
\end{align}
for any $T\gg1$ with compactly supported function $\psi\in\ml{C}^2\big([0,T]\times\mb{R}^n\big)$ such that $\psi(T,x)=\psi_t(T,x)=0$. Then, we define the test function by separating the variables fulfilling
\begin{align*}
	\psi(t,x):=D_{t|T}^{1-\gamma}\left(\widetilde{\psi}(t,x)\right)\ \ \mbox{with}\ \ \widetilde{\psi}(t,x)=\varphi_T(x)^{\ell}w(t),
\end{align*}
where $w(t)$ was defined in \eqref{Defn_w(t)} and
\begin{align}\label{TEST_FUNCTION_NEW}
\varphi_T(x):=\varphi\left(\frac{R|x|}{T^d}\right)
\end{align}
  with positive parameters $R$ and $d$ to be determined later in each situation. Here, $\varphi$ is a radial non-increasing test function with $\varphi\in\ml{C}^{\infty}(\mb{R})$ fulfilling
\begin{align*}
	\varphi(r):=\begin{cases}
		1&\mbox{if}\ \ 0\leqslant r\leqslant 1,\\
		0&\mbox{if}\ \ r>2,
	\end{cases}
\end{align*}
moreover, $|\varphi^\prime(r)|\lesssim r^{-1}$.
\begin{remark}
Let us briefly talk about the roles of two parameters in \eqref{TEST_FUNCTION_NEW}. As usual, the parameter $R$ is chosen by the constant $1$ in the subcritical case. We may take it as a large number in the critical case in order to claim some contradictions for getting blow-up. For another, the time scaling parameter $d$ allows us to enlarge admissible range of blow-up (i.e. the range for $p$), which exerts a new influence on our result.
\end{remark}

To begin with, let us define
\begin{align*}
	I_T&:=\int_0^T\int_{B_{2T^d/R}}|u(t,x)|^pg(t)\widetilde{\psi}(t,x)\mathrm{d}x\mathrm{d}t,\\
	I_0&:=\int_{B_{2T^d/R}}\big(u_1(x)+(\mu-1)u_0(x)\big)\varphi_T(x)^{\ell}\mathrm{d}x.
\end{align*}
We notice that
\begin{align*}
	\int_0^T\int_{\mb{R}^n}N_{\gamma,p}[u](t,x)g(t)\psi(t,x)\mathrm{d}x\mathrm{d}t&=\int_0^T\int_{\mb{R}^n}I_{0|t}^{1-\gamma}(|u|^p)(t,x)g(t)D_{t|T}^{1-\gamma}\left(\widetilde{\psi}(t,x)\right)\mathrm{d}x\mathrm{d}t\\
	&\geqslant\int_0^T\int_{\mb{R}^n}D_{0|t}^{1-\gamma}I_{0|t}^{1-\gamma}(|u|^pg)(t,x)\widetilde{\psi}(t,x)\mathrm{d}x\mathrm{d}t=I_T
\end{align*}
by employing $g(t)\geqslant g(\tau)$ for $\tau\in[0,t]$, as well as Propositions \ref{Prop_3.1} and \ref{Prop_3.2}.\\
Then, using \eqref{Weak_Solution_01} and the derived estimates in the above, we get
\begin{align*}
	I_T + I_0
	&\lesssim-\frac{T^{-2+\gamma}}{\mu-1}\int_{B_{2T^d/R}}u_0(x)\varphi_T(x)^{\ell}\mathrm{d}x+ \int_0^T\int_{B_{2T^d/R}}|u(t,x)|g(t)\varphi_T(x)^{\ell}\left|D_{t|T}^{3-\gamma}w(t)\right|\mathrm{d}x\mathrm{d}t\\
	&\quad+|\mu-2|\int_0^T\int_{B_{2T^d/R}}|u(t,x)|\varphi_T(x)^{\ell}\left|D_{t|T}^{2-\gamma}w(t)\right|\mathrm{d}x\mathrm{d}t\\
	&\quad+\int_0^T\int_{B_{2T^d/R}}|u(t,x)|g(t)\left|\Delta\varphi_T(x)^{\ell}\right|\left|D_{t|T}^{1-\gamma}w(t)\right|\mathrm{d}x\mathrm{d}t\\
	&=:-\frac{T^{-2+\gamma}}{\mu-1}\int_{B_{2T^d/R}}u_0(x)\varphi_T(x)^{\ell}\mathrm{d}x+J_{1,T}+J_{2,T}+J_{3,T},
\end{align*}
where we used Proposition \ref{Prop_3.2} and \eqref{Frac_Cal_01}. Next, from Young's inequality, we obtain
\begin{align*}
	J_{1,T}&\lesssim\frac{I_{T}}{6}+\int_0^T\int_{B_{2T^d/R}}g(t)\varphi_T(x)^{\ell}w(t)^{-\frac{1}{p-1}}\left|D_{t|T}^{3-\gamma}w(t)\right|^{p'}\mathrm{d}x\mathrm{d}t,\\
	J_{2,T}&\lesssim\frac{I_{T}}{6}+|\mu-2|^{p'}\int_0^T\int_{B_{2T^d/R}}g(t)^{-\frac{1}{p-1}}\varphi_T(x)^{\ell}w(t)^{-\frac{1}{p-1}}\left|D_{t|T}^{2-\gamma}w(t)\right|^{p'}\mathrm{d}x\mathrm{d}t,\\
	J_{3,T}&\lesssim\frac{I_T}{6}+\int_0^t\int_{B_{2T^d/R}}g(t)\varphi_T(x)^{\ell-2p'}w(t)^{-\frac{1}{p-1}}\left(|\Delta\varphi_T(x)|^{p'}+|\nabla\varphi_T(x)|^{2p'}\right)\left|D_{t|T}^{1-\gamma}w(t)\right|^{p'}\mathrm{d}x\mathrm{d}t.
\end{align*}
In the last estimate, we used
\begin{align*}
	\Delta\varphi_T(x)^{\ell}=\ell\varphi_T(x)^{\ell-1}\Delta\varphi_T(x)+\ell(\ell-1)\varphi_T(x)^{\ell-2}|\nabla\varphi_T(x)|^{2}.
\end{align*}
Summarizing the derived estimates, we conclude
\begin{align}
	I_T+I_0&\lesssim-T^{-2+\gamma}\int_{B_{2T^d/R}}u_0(x)\varphi_T(x)^{\ell}\mathrm{d}x+ \int_0^Tg(t)w(t)^{-\frac{1}{p-1}}\left|D_{t|T}^{3-\gamma}w(t)\right|^{p'}\mathrm{d}t\ \int_{B_{2T^d/R}}\varphi_T(x)^{\ell}\mathrm{d}x\notag\\
	&\quad+|\mu-2|^{p'}\int_0^T\big(g(t)w(t)\big)^{-\frac{1}{p-1}}\left|D_{t|T}^{2-\gamma}w(t)\right|^{p'}\mathrm{d}t\ \int_{B_{2T^d/R}}\varphi_T(x)^{\ell}\mathrm{d}x\notag\\
	&\quad+\int_0^Tg(t)w(t)^{-\frac{1}{p-1}}\left|D_{t|T}^{1-\gamma}w(t)\right|^{p'}\mathrm{d}t\ \int_{B_{2T^d/R}}\varphi_T(x)^{\ell-2p'}\left(|\Delta\varphi_T(x)|^{p'}+|\nabla\varphi_T(x)|^{2p'}\right)\mathrm{d}x.\label{Eq_Summarize}
\end{align}
Obviously, the value of $\mu$ influences on the last inequality due to the fact that when $\mu=2$, the third term on the right-hand side of \eqref{Eq_Summarize} will be vanishing.

%
In the subcritical case, we may consider $R=1$ in the test function $\varphi_T(x)$ so that $\varphi_T(x):=\varphi(|x|/T^d)$. By applying Lemma \ref{Lem_Fraction_1} in \eqref{Eq_Summarize}, it yields
\begin{align}\label{Sub_Crit_Blow_Fuji}
	I_T +I_0\lesssim T^{-2+\gamma}\|u_0\|_{L^1(\mb{R}^n)}+ T^{2-(3-\gamma)p'+nd}+T^{2-(1-\gamma)p'+nd-2dp'}+|\mu-2|^{p'}\ml{D}_p(T)T^{nd-(2-\gamma)p'}.
\end{align}
Generally speaking, if the powers of $T$ in the previous terms are negative under some conditions of the exponent $p$, we may take $T\to\infty$ so that the right-hand side of \eqref{Sub_Crit_Blow_Fuji} tending to zero. Thanks to the Lebesgue dominated convergence theorem, it yields
\begin{align*}
0<\int_{\mb{R}^n}\big(u_1(x)-(\mu-1)u_0(x)\big)\mathrm{d}x\leqslant0;
\end{align*}
contradiction. In the limit case (it means the maximum of all powers of $T$ in \eqref{Sub_Crit_Blow_Fuji} is zero), we may take $1\ll R<T$ such that $T$ and $R$ do not tend simultaneously to infinity. Consequently, there exists a positive constant $\widetilde{c}$ independent of $T$ such that
\begin{align*}
	\int_0^{\infty}\int_{\mb{R}^n}|u(t,x)|^p g(t)\mathrm{d}x\mathrm{d}t\leqslant\widetilde{c},
\end{align*}
which implies that
\begin{align}\label{Min_J3_Pre}
	\int_0^T\int_{B_{2T/R}\backslash B_{T/R}}|u(t,x)|^pg(t)\widetilde{\psi}(t,x)\mathrm{d}x\mathrm{d}t\to0\ \ \mbox{as}\ \ T\to\infty.
\end{align}
At this moment, we apply H\"older's inequality instead of Young's inequality to estimate the worst term among $J_{1,T}$, $J_{2,T},J_{3,T}$, which claims by standard test function argument that
\begin{align*}
\int_{\mb{R}^n}\big(u_1(x)-(\mu-1)u_0(x)\big)\mathrm{d}x\lesssim R^{-\tilde{\alpha}},
\end{align*}
with $\tilde{\alpha}>0$. It  implies a contradiction when $R$ is suitably large. This may complete our desired the blow-up result.

We should emphasize from the last discussion that our desired range for the exponent $p$ should satisfies $1<p\leqslant\frac{n}{(n-2)_+}$ from local (in time) existence and
\begin{align*}
\begin{cases}
2-(3-\gamma)p'+nd\leqslant0,\\
2-(1-\gamma)p'+nd-2dp'\leqslant 0,
\end{cases}
\end{align*}
as well as if $\mu\neq2$, then
\begin{align*}
\begin{cases}
\frac{p-2}{p-1}+nd-(2-\gamma)p'\leqslant0&\mbox{when}\ \ p>2,\\
nd-(2-\gamma)p'<0&\mbox{when}\ \ p=2,\\
nd-(2-\gamma)p'\leqslant 0&\mbox{when}\ \ p<2.
\end{cases}
\end{align*}
These ranges heavily rely on the suitable choice of $d$ by some elementary computations. Hence, providing that the last restrictions are fulfilled, we immediately arrive at blow-up of solutions. For the sake of readability, we next will state some results only.

Let us first consider the case $\mu\in(1,\infty)\backslash\{2\}$. Under this situation, there are three cases $n=1,2$, $n=3$, $n\geqslant 4$ that need to be considered since the comparison $2$ and $\frac{n}{(n-2)_+}$. Recalling the notations
\begin{align*}
	p_1(n):=1+\frac{3-\gamma}{n-1+\gamma}\ \ \mbox{as well as} \ \  p_2(n,d_0):=1+\frac{2-\gamma}{nd_0-2+\gamma},
\end{align*}
moreover,
\begin{align*}
	d_0(n):=\frac{1}{4}+\sqrt{\frac{1}{16}+\frac{2-\gamma}{n}},
\end{align*}
we have the following classifications.
\begin{itemize}
	\item When $n=1,2$, let us choose $d=1$ if $p>2$, and $d=d_0(n)$ if $p\leqslant 2$. Then, the blow-up result holds if $p\leqslant p_1(n)$.
	\item When $n=3$, let us choose $d=1$ if $p>2$, and $d=d_0(3)$ if $p\leqslant 2$. Then, the blow-up result holds if $p\leqslant\min\{ p_1(3),p_2(3,d_0) \}$.
	\item When $n\geqslant 4$, let us choose $d=d_0(n)$. Then, the blow-up result holds if $p\leqslant \min\{p_2(n,d_0),\frac{n}{n-2}\}$.
\end{itemize}
For the rest case $\mu=2$, the discussion becomes easier. By selecting $d=1$, then the blow-up result holds if $p\leqslant\min\{p_1(n),\frac{n}{(n-2)_+}\}$.

\subsection{Blow-up result in the case $\mu\in(0,1]$}
In this case, we will apply the similar idea as the last subsection so that we only sketch the proof. Let us replace the time-dependent function $g=g(t)$ with $g(0)=1$ in \eqref{function g} by $g(t):=(1+t)^{\mu}$, which results that
\begin{align*}
	gN_{\gamma,p}[u]=(gu)_{tt}-\Delta(gu)-(g'u)_t.
\end{align*}
By choosing the same test function $\psi=\psi(t,x)$ as the previous result for $\mu>1$, we may compute
\begin{align*}
	I_T\lesssim-T^{-2+\gamma}\int_{B_{2T^d/R}}u_0(x)\varphi_T(x)^{\ell}\mathrm{d}x-T^{-1+\gamma}\int_{B_{2T^d/R}}u_1(x)\varphi_T(x)^{\ell}\mathrm{d}x+\widetilde{J}_{1,T}+\widetilde{J}_{2,T}+\widetilde{J}_{3,T},
\end{align*}
where we defined
\begin{align*}
	\widetilde{J}_{1,T}&:=\int_0^T\int_{B_{2T^d/R}}|u(t,x)|g(t)|\psi_{tt}(t,x)|\mathrm{d}x\mathrm{d}t,\\
	\widetilde{J}_{2,T}&:=\int_0^T\int_{B_{2T^d/R}}|u(t,x)|g'(t)|\psi_t(t,x)|\mathrm{d}x\mathrm{d}t,\\
	\widetilde{J}_{3,T}&:=\int_0^T\int_{B_{2T^d/R}}|u(t,x)|g(t)|\Delta\psi(t,x)|\mathrm{d}x\mathrm{d}t.
\end{align*}
By using Young's inequality, one may find
\begin{align*}
	I_T&\lesssim T^{-1+\gamma}\left(\|u_0\|_{L^1(\mb{R}^n)}+\|u_1\|_{L^1(\mb{R}^n)}\right)+\int_0^T\int_{B_{2T^d/R}}g(t)\varphi_T(x)^{\ell}w(t)^{-\frac{1}{p-1}}\left|D_{t|T}^{3-\gamma}w(t)\right|^{p'}\mathrm{d}x\mathrm{d}t\\
	&\quad+\int_0^T\int_{B_{2T^d/R}}g(t)^{-\frac{1}{p-1}}g'(t)^{p'}\varphi_T(x)^{\ell}w(t)^{-\frac{1}{p-1}}\left|D_{t|T}^{2-\gamma}w(t)\right|^{p'}\mathrm{d}x\mathrm{d}t\\
	&\quad+\int_0^T\int_{B_{2T^d/R}}g(t)\varphi_T(x)^{\ell-2p'}w(t)^{-\frac{1}{p-1}}\left(|\Delta\varphi_T(x)|^{p'}+|\nabla\varphi_T(x)|^{2p'}\right)\left|D_{t|T}^{1-\gamma}w(t)\right|^{p'}\mathrm{d}x\mathrm{d}t.
\end{align*}
Similarly to Lemma \ref{Lem_Fraction_1}, the next inequalities can be demonstrated without additional difficulty:
\begin{align*}
	\int_0^T(1+t)^{\mu}w(t)^{-\frac{1}{p-1}}\left|D_{t|T}^{k+\alpha}w(t)\right|^{\frac{p}{p-1}}\mathrm{d}t&\lesssim T^{\mu+1-\frac{(k+\alpha)p}{p-1}},\\
	\int_0^T(1+t)^{\mu-\frac{p}{p-1}}w(t)^{-\frac{1}{p-1}}\left|D_{t|T}^{k+\alpha}w(t)\right|^{\frac{p}{p-1}}\mathrm{d}t&\lesssim \widetilde{\ml{D}}_p(T)T^{-\frac{(k+\alpha)p}{p-1}},
\end{align*}
with $k\in\mb{N}$ and $\alpha\in(0,1)$, where
\begin{align*}
	\widetilde{\ml{D}}_p(T):=\begin{cases}
		T^{\mu-\frac{1}{p-1}}&\mbox{if}\ \ p>1+\frac{1}{\mu},\\
		\ln T&\mbox{if}\ \ p=1+\frac{1}{\mu},\\
		1&\mbox{if}\ \ p<1+\frac{1}{\mu},
	\end{cases}
\end{align*}
For this reason, we proved
\begin{align*}
	I_T&\lesssim T^{-1+\gamma}\left(\|u_0\|_{L^1(\mb{R}^n)}+\|u_1\|_{L^1(\mb{R}^n)}\right)+T^{\mu+1-(1-\gamma)p'+nd-2dp'}+T^{nd+\mu+1-(3-\gamma)p'}+	\widetilde{\ml{D}}_p(T)T^{nd-(2-\gamma)p'}
\end{align*}
By the same procedure as the last part, in order to derive contradiction by taking $T\to\infty$, we need to guarantee $1<p\leqslant\frac{n}{(n-2)_+}$ from local (in time) existence and
\begin{align*}
\begin{cases}
\mu+1-(1-\gamma)p'+nd-2dp'\leqslant0,\\
nd+\mu+1-(3-\gamma)p'\leqslant0,
\end{cases}
\end{align*}
as well as
\begin{align*}
\begin{cases}
\mu-\frac{1}{p-1}+nd-(2-\gamma)p'\leqslant0&\mbox{when}\ \ p>1+\frac{1}{\mu},\\
nd-(2-\gamma)p'<0&\mbox{when}\ \ p=1+\frac{1}{\mu},\\
nd-(2-\gamma)p'\leqslant0&\mbox{when}\ \ p<1+\frac{1}{\mu}.
\end{cases}
\end{align*}
To do so, we have to compare $1+\frac{1}{\mu}$, $p_3(n)$ and $\frac{n}{(n-2)_+}$ by dividing into four cases $n=1$, $n=2$, $n=3$ and $n\geqslant 4$.  Recalling the notations
\begin{align*}
	p_3(n):=1+\frac{3-\gamma}{(n+\mu+\gamma-2)_+}\ \ \mbox{as well as} \ \  p_4(n,d_1):=1+\frac{2-\gamma}{nd_1-2+\gamma},
\end{align*}
moreover,
\begin{align*}
	d_1(n):=\frac{1}{4}+\sqrt{\frac{1}{16}+\frac{(\mu+1)(2-\gamma)}{2n}},
\end{align*}
we have the following classifications.
\begin{itemize}
	\item When $n=1$, let us choose $d=1$ if $p>1+\frac{1}{\mu}$, and $d=d_1(1)$ if $p\leqslant 1+\frac{1}{\mu}$. Then, the blow-up result holds if $1<p\leqslant p_3(1)$.
	\item When $n=2$, let us choose $d=1$ if $p>1+\frac{1}{\mu}$, and $d=d_1(2)$ if $p\leqslant 1+\frac{1}{\mu}$. Then, the blow-up result holds if $1<p\leqslant \min\{p_3(2),p_4(2,d_1)\}$.
	\item When $n=3$, let us choose $d=1$ if $p>1+\frac{1}{\mu}$, and $d=d_1(3)$ if $p\leqslant 1+\frac{1}{\mu}$. Then, the blow-up result holds if $1<p\leqslant \min\{p_3(3),p_4(3,d_1),3 \}$.
	\item When $n\geqslant 4$, let us simply choose $d=d_1(n)$. Then, the blow-up result holds providing that $1<p\leqslant \min\{p_4(n,d_1),\frac{n}{n-2}\}$.
\end{itemize}

\section{Proof of Theorem \ref{Thm_Blow_Up_mu_0} via generalized Kato's type lemma}
\subsection{Derivation of generalized Kato's type lemma}
In this subsection, we will develop generalized Kato's type lemma of the integral type to be used later, whose proof is based on the iteration argument.
\begin{lemma}\label{Generalized Kato lem}
	Let us assume that $p>1$ and $\ml{F}(t)\in\ml{C}[0,T)$ such that 
	\begin{align}
		\ml{F}(t)&\geqslant K_0(1+t)^{-\alpha_0}(t-T_0)^{\beta_0}\label{FLB 01}\\
		\ml{F}(t)&\geqslant \widetilde{K}_0(1+t)^{-a_0}\int_{T_0}^t(1+\eta)^{a_1}\int_{T_0}^{\eta}(1+s)^{a_2}\int_{T_0}^s(1+\tau)^{a_3}|\ml{F}(\tau)|^p\mathrm{d}\tau\mathrm{d}s\mathrm{d}\eta\label{FRAM 01}
	\end{align}
	for any $t\geqslant T_0\geqslant 0$, where $\alpha_0,\beta_0,a_0,\dots,a_3$ are nonnegative constants, and $K_0,\widetilde{K}_0$ are positive constants. If these parameters fulfill the relation
	\begin{align}\label{Blow-Up condition}
		(\beta_0-\alpha_0)(p-1)+a_1+a_2+a_3+3-a_0>0,
	\end{align}
	then the functional $\ml{F}(t)$ blows up in finite time.
\end{lemma}
\begin{remark}
	In the case $a_k<0$ for $k=1,2,3$ in \eqref{FRAM 01}, we still can prove blow-up for the functional $\ml{F}(t)$. For example, when $a_3<0$, from $(1+\tau)^{a_3}\geqslant (1+t)^{a_3}$ for any $\tau\in[0,t]$, one has
	\begin{align*}
		\ml{F}(t)\geqslant \widetilde{K}_0(1+t)^{-(a_0-a_3)}\int_{T_0}^t(1+\eta)^{a_1}\int_{T_0}^{\eta}(1+s)^{a_2}\int_{T_0}^s|\ml{F}(\tau)|^p\mathrm{d}\tau\mathrm{d}s\mathrm{d}\eta.
	\end{align*}
	From Lemma \ref{Generalized Kato lem}, we can get blow-up of the functional $\ml{F}(t)$ if $(\beta_0-\alpha_0)(p-1)+a_1+a_2+3-(a_0-a_3)>0$, which is exactly the same as \eqref{Blow-Up condition}.
\end{remark}
\begin{proof}
	Motivated by the first lower bound \eqref{FLB 01}, we will demonstrate the functional $\ml{F}(t)$ having the following lower bound estimates:
	\begin{align}\label{Seq_Iter_F}
		\ml{F}(t)\geqslant K_j(1+t)^{-\alpha_j}(t-T_0)^{\beta_j},
	\end{align}
	for any $t\geqslant T_0$, where the sequences $\{K_j\}_{j\in\mb{N}}$, $\{\alpha_j\}_{j\in\mb{N}}$ and $\{\beta_j\}_{j\in\mb{N}}$ consist of nonnegative real numbers to be determined later. Clearly from our observation, the initial case when $j=0$ is given by \eqref{FLB 01}. To prove \eqref{Seq_Iter_F} by deriving the sequences, we may use an iteration procedure. Precisely, we assume \eqref{Seq_Iter_F} holding for $j$ and it still remains to do the inductive step, i.e. we will show that \eqref{Seq_Iter_F} is also valid for $j+1$.
	
	First of all, we combine \eqref{Seq_Iter_F} with \eqref{FRAM 01} to get immediately
	\begin{align*}
		\ml{F}(t)&\geqslant K_j^p\widetilde{K}_0(1+t)^{-a_0}\int_{T_0}^t(1+\eta)^{a_1}\int_{T_0}^{\eta}(1+s)^{a_2}\int_{T_0}^s(1+\tau)^{a_3-p\alpha_j}(\tau-T_0)^{p\beta_j}\mathrm{d}\tau\mathrm{d}s\mathrm{d}\eta\\
		&\geqslant K_j^p\widetilde{K}_0(1+t)^{-a_0-p\alpha_j}\int_{T_0}^t(1+\eta)^{a_1}\int_{T_0}^{\eta}(1+s)^{a_2}\int_{T_0}^s(\tau-T_0)^{a_3+p\beta_j}\mathrm{d}\tau\mathrm{d}s\mathrm{d}\eta\\
		&\geqslant \frac{K_j^p\widetilde{K}_0}{a_3+1+p\beta_j}(1+t)^{-a_0-p\alpha_j}\int_{T_0}^t(1+\eta)^{a_1}\int_{T_0}^{\eta}(s-T_0)^{a_2+a_3+1+p\beta_j}\mathrm{d}s\mathrm{d}\eta\\
		&\geqslant \frac{K_j^p\widetilde{K}_0}{(a_3+1+p\beta_j)(a_2+a_3+2+p\beta_j)}(1+t)^{-a_0-p\alpha_j}\int_{T_0}^t(\eta-T_0)^{a_1+a_2+a_3+2+p\beta_j}\mathrm{d}\eta\\
		&\geqslant \frac{K_j^p\widetilde{K}_0}{(a_1+a_2+a_3+3+p\beta_j)^3}(1+t)^{-a_0-p\alpha_j}(t-T_0)^{a_1+a_2+a_3+3+p\beta_j}
	\end{align*}
	for all $t\geqslant T_0$ and we used nonnegativities of $a_1,a_2,a_3$. Therefore, the desired estimate \eqref{Seq_Iter_F} for $j+1$ is concluded, provided that the recursive relations
	\begin{align*}
		K_{j+1}:= \frac{K_j^p\widetilde{K}_0}{(a_1+a_2+a_3+3+p\beta_j)^3},
	\end{align*}
	and $\alpha_{j+1}:= a_0+p\alpha_j$, $\beta_{j+1}:= a_1+a_2+a_3+3+p\beta_{j}$ are fulfilled.
	
	To determine the estimate for the multiplicative constant $K_j$ from the below, we should derive the explicit representation for $\alpha_j$ and $\beta_j$ in the first place. From the relations
	\begin{align}\label{Eq_Relation_alpha_beta}
		\alpha_{j}= a_0+p\alpha_{j-1}\ \ \mbox{and} \ \ \beta_{j}= a_1+a_2+a_3+3+p\beta_{j-1},
	\end{align}
	we can deduce by iteration calculations
	\begin{align}
		\alpha_j&=p^j\alpha_0 +a_0\sum\limits_{k=0}^{j-1}p^{k}=\left(\alpha_0+\frac{a_0}{p-1}\right)p^j-\frac{a_0}{p-1},\label{Seq.alpha}\\
		\beta_j&=p^j\beta_0 +(a_1+a_2+a_3+3)\sum\limits_{k=0}^{j-1}p^{k}=\left(\beta_0+\frac{a_1+a_2+a_3+3}{p-1}\right)p^j-\frac{a_1+a_2+a_3+3}{p-1}.\label{Seq.beta}
	\end{align}
	One may observe that
	\begin{align*}
		(a_1+a_2+a_3+3+p\beta_{j-1})^3=\beta_j^3\leqslant \left(\beta_0+\frac{a_1+a_2+a_3+3}{p-1}\right)^3p^{3j},
	\end{align*}
	where we used \eqref{Eq_Relation_alpha_beta} and \eqref{Seq.beta}.
	
	Then, it follows that
	\begin{align}\label{Eq_Seq_Kj}
		K_j\geqslant \widetilde{K}_0\left(\beta_0+\frac{a_1+a_2+a_3+3}{p-1}\right)^{-3}p^{-3j}K_{j-1}^p=: D p^{-3j}K_{j-1}^p
	\end{align}
	for any $j\in\mb{N}$, with a suitable constant $D>0$.
	
	In order to hit our mark, we employ the logarithmic function to both sides of \eqref{Eq_Seq_Kj} to get
	\begin{align*}
		\log K_j&\geqslant p^j\log K_0-3\log p\sum\limits_{k=0}^{j-1}(j-k)p^k+\log D\sum\limits_{k=0}^{j-1}p^k\\
		&\geqslant p^j\left(\log K_0-\frac{3p\log p}{(p-1)^2}+\frac{\log D}{p-1}\right)+\frac{3j\log p}{p-1}+\frac{3p\log p}{(p-1)^2}-\frac{\log D}{p-1}
	\end{align*}
	for any $j\in\mb{N}$, where the next well-known formula:
	\begin{align} \label{identity sum (j-k)p^k}
		\sum\limits_{k=0}^{j-1}(j-k)p^k = \frac{1}{p-1}\left(\frac{p^{j+1}-p}{p-1}-j\right)
	\end{align} was applied. Let us choose $j_0=j_0(p,a_1,a_2,a_3)$ to be the smallest positive integer such that
	\begin{align*}
		j_0\geqslant\frac{\log D}{3\log p}-\frac{p}{p-1}.
	\end{align*}
	Taking into account $j\geqslant j_0$ the inequality holds
	\begin{align}\label{Est.Kj}
		\log K_j\geqslant p^j\left(\log K_0-\frac{3p\log p}{(p-1)^2}+\frac{\log D}{p-1}\right)=:p^j\log E_0
	\end{align}
	with a suitable constant $E_0=E_0(p,a_1,a_2,a_3)>0$.

	Finally, let us associate \eqref{Seq_Iter_F}, \eqref{Seq.alpha}, \eqref{Seq.beta} with \eqref{Est.Kj}. By this way, it yields
	\begin{align*}
		\ml{F}(t)&\geqslant \exp\left[p^j\left(\log E_0-\left(\alpha_0+\frac{a_0}{p-1}\right)\log(1+t)+\left(\beta_0+\frac{a_1+a_2+a_3+3}{p-1}\right)\log (t-T_0)\right)\right]\\
		&\quad\times(1+t)^{\frac{a_0}{p-1}} (t-T_0)^{-\frac{a_1+a_2+a_3+3}{p-1}}
	\end{align*}
	for any $j\geqslant j_0$ and $t\geqslant T_0$. Let us assume $t\geqslant \max\{1,2T_0\}$, which implies $\log(1+t)\leqslant \log (2t)$ and $\log(t-T_0)\geqslant\log(t/2)$. Therefore, from the above result, we may write
	\begin{align}\label{Lower.Bound.F.Final}
		\ml{F}(t)&\geqslant \exp \left[p^j\log\left(E_0\,2^{-\alpha_0-\beta_0-\frac{a_0+a_1+a_2+a_3+3}{p-1}}\,t^{\beta_0-\alpha_0+\frac{a_1+a_2+a_3+3-a_0}{p-1}}\right)\right]\notag\\
		&\quad\times(1+t)^{\frac{a_0}{p-1}} (t-T_0)^{-\frac{a_1+a_2+a_3+3}{p-1}}
	\end{align}
	for any $j\geqslant j_0$ and $t\geqslant \max\{1,2T_0\}$. With our assumption on $p$ such that \eqref{Blow-Up condition} holds, we claim that the power for $t$ in the exponential term of \eqref{Lower.Bound.F.Final} is positive. Thus, we may find 
	\begin{align*}
		\log\left(E_0\,2^{-\alpha_0-\beta_0-\frac{a_0+a_1+a_2+a_3+3}{p-1}}\,t^{\beta_0-\alpha_0+\frac{a_1+a_2+a_3+3-a_0}{p-1}}\right)>0
	\end{align*}
	for suitably large $t\geqslant\max\{1,2T_0\}$. Letting $j\to\infty$, we observe blow-up phenomenon of the functional $\ml{F}(t)$. Thus, the proof is complete. 
\end{proof}

\subsection{Blow-up result in the case $\mu\in(0,\infty)$}
Let us introduce a time-dependent functional with respect to the solution by
\begin{align*}
 	F(t):=\int_{\mb{R}^n}u(t,x)\mathrm{d}x.	
\end{align*}
 We now take the test function in \eqref{Eq.Defn.Energy.Solution_2} satisfying $\psi\equiv1$ over the set $\{(s,x)\in[0,t]\times B_{R+s}\}$. It immediately results 
\begin{align}\label{Eq_Sec_4_01}
	&\int_{\mb{R}^n}u_t(t,x)\mathrm{d}x-\int_{\mb{R}^n}u_1(x)\mathrm{d}x+\int_0^t\frac{\mu }{1+s}\int_{\mb{R}^n}u_t(s,x)\mathrm{d}x\mathrm{d}s\notag\\
	&\qquad=c_\gamma\int_0^t\int_{\mb{R}^n}\int_0^s(s-\tau)^{-\gamma}|u(\tau,x)|^p\mathrm{d}\tau\mathrm{d}x\mathrm{d}s,
\end{align}
and differentiate \eqref{Eq_Sec_4_01} with respect to the time variable to conclude
\begin{align*}
	\int_{\mb{R}^n}u_{tt}(t,x)\mathrm{d}x+\frac{\mu }{1+t}\int_{\mb{R}^n}u_t(t,x)\mathrm{d}x=c_\gamma\int_{\mb{R}^n}\int_0^t(t-\tau)^{-\gamma}|u(\tau,x)|^p\mathrm{d}\tau\mathrm{d}x.
\end{align*}
Clearly, the previous equality can be reformulated as
\begin{align}\label{Eq_Sec_4_02}
	(1+t)^{-\mu}\big(F'(t)(1+t)^{\mu}\big)'=c_{\gamma}\int_0^t(t-\tau)^{-\gamma}\int_{\mb{R}^n}|u(\tau,x)|^p\mathrm{d}x\mathrm{d}\tau.
\end{align}
Then, multiplying  \eqref{Eq_Sec_4_02} by $(1+t)^{\mu}$ and integrating the resultant over $[0,t]$, we may see
\begin{align}\label{Eq_Pre_Frame_F3}
	F(t)\geqslant c_{\gamma}\int_0^t(1+\eta)^{-\mu}\int_0^\eta(1+s)^{\mu}\int_0^s(s-\tau)^{-\gamma}\int_{\mb{R}^n}|u(\tau,x)|^p\mathrm{d}x\mathrm{d}\tau\mathrm{d}s\mathrm{d}\eta\geqslant0,
\end{align}
where we used nonnegativities of $u_0$ and $u_1$ implying $F(0)\geqslant0$ as well as $F'(0)\geqslant0$.\\
Furthermore, by using H\"older's inequality and the support condition given by finite proposition speed, one has
\begin{align}\label{Eq_Sec_4_03}
	\int_{\mb{R}^n}|u(\tau,x)|^p\mathrm{d}x=\int_{B_{R+\tau}}|u(\tau,x)|^p\mathrm{d}x\geqslant C_0(1+\tau)^{-n(p-1)}|F(\tau)|^p,
\end{align}
with a positive constant $C_0=C_0(n,R,p)$. For this reason, the desired inequality \eqref{FRAM 01} is constructed by plugging \eqref{Eq_Sec_4_03} into \eqref{Eq_Pre_Frame_F3} so that
\begin{align}\label{Eq_Iteration_Frame}
	F(t)\geqslant C_0\,c_{\gamma}(1+t)^{-\mu-\gamma-n(p-1)}\int_0^t\int_0^\eta(1+s)^{\mu}\int_0^s|F(\tau)|^p\mathrm{d}\tau\mathrm{d}s\mathrm{d}\eta
\end{align}
for any $t\geqslant0$.

The main approach of our proof is based on Lemma \ref{Generalized Kato lem}, which needs the lower bound estimate for the functional. We are motived by the paper \cite{Yordanov-Zhang2006}, in other words, we introduce the test function $\Phi=\Phi(x)$ such that
\begin{align} \label{def eigenfunction laplace op}
	\Phi(x) :=
	\begin{cases}
		\mathrm{e}^{x}+\mathrm{e}^{-x}&\mbox{if} \ \  n=1,\\
		\displaystyle{\int_{\mathbb{S}^{n-1}} \mathrm{e}^{x\cdot \omega}\mathrm{d} \sigma_\omega}&\mbox{if} \ \ n\geqslant 2,
	\end{cases}
\end{align}
where $\mb{S}^{n-1}$ is the $n-1$ dimensional sphere. The above function is a positive smooth and fulfills the properties
\begin{align*}
	\Delta \Phi =\Phi, \ \ \mbox{as well as}\ \  \Phi (x) \sim |x|^{-\frac{n-1}{2}} \, \mathrm{e}^{|x|} \ \  \mbox{as} \ \  |x|\to \infty.
\end{align*}
According to \cite{Tu-Lin201701,Palmieri-Tu2019}, we recall the modified Bessel function of the second kind by
\begin{align*}
	\ml{K}_{\nu}(t):=\int_0^t\exp\left(-t\cosh z\right)\cosh (\nu z)\mathrm{d}z
\end{align*}
for any $\nu\in\mb{R}$, which solves the $\nu$-dependent second-order ODE
\begin{align*}
	\left(t^2\frac{\mathrm{d}^2}{\mathrm{d}t^2}+t\frac{\mathrm{d}}{\mathrm{d}t}-(t^2+\nu^2)\right)\ml{K}_{\nu}(t)=0 \ \ \mbox{with} \ \ 
	\ml{K}_{\nu}(0)=0.
\end{align*}
Recalling \cite{Erdelyi-Magnus-Oberhettinger-Tricomi1953}, the asymptotic behavior of it is showed for $t\to \infty$ as
\begin{align*}
	\ml{K}_{\nu}(t)=\sqrt{\frac{\pi}{2t}}\,\mathrm{e}^{-t}\left(1+\ml{O}(t^{-1})\right).
\end{align*}
Its derivative fulfills
\begin{align*}
	\frac{\mathrm{d}}{\mathrm{d}t}\ml{K}_{\nu}(t)=-\ml{K}_{\nu+1}(t)+\frac{\nu}{t}\ml{K}_{\nu}(t)=-\frac{1}{2}\big(\ml{K}_{\nu+1}(t)+\ml{K}_{\nu-1}(t)\big).
\end{align*}
Setting the auxiliary function
\begin{align*}
	\lambda(t):=(1+t)^{\frac{\mu+1}{2}}\ml{K}_{(\mu-1)/2}(1+t),
\end{align*}
we observe that it is the solutions to the following differential equation:
\begin{align*}
	\left((1+t)^2\frac{\mathrm{d}^2}{\mathrm{d}t^2}-\mu(1+t)\frac{\mathrm{d}}{\mathrm{d}t}+\left(\mu-(1+t)^2\right)\right)\lambda(t)=0\ \ \mbox{with} \ \
	\lambda(0)=\ml{K}_{(\mu-1)/2}(1), \ \ \lambda(\infty)=0.
\end{align*}
Let us now introduce the test function $\Psi=\Psi(t,x)$ with separate variables by
\begin{align}\label{Defn_Test_Fun_Psi}
	\Psi(t,x):=\lambda(t)\Phi(x).
\end{align}
Indeed, we find that \cite[Lemma 2.1]{Tu-Lin201701} or \cite[Lemma 2.1]{Palmieri-Tu2019} is still valid for \eqref{Eq_Semi_Wave_Scale_Memory} due to the fact that the proof of such lemma is independent of nonnegative nonlinearity $N_{\gamma,p}[u]\geqslant 0$ for any $\gamma\in(0,1)$ and $p>1$. Consequently, it holds
\begin{align}\label{Eq_Pre_Lower_Bound}
	\int_{\mb{R}^n}|u(t,x)|^p\mathrm{d}x\geqslant C_1 (1+t)^{n-1-\frac{n+\mu-1}{2}p}
\end{align}
for any $t\geqslant T_0$, where $T_0$ is a large number independent of $u_0,u_1$ and $C_1=C_1(u_0,u_1,n,p,\mu,R,\Phi)$ is a positive constant. Eventually, combining \eqref{Eq_Pre_Lower_Bound} with \eqref{Eq_Pre_Frame_F3} and using $(1+\tau)\geqslant (\tau-T_0)$ yield
\begin{align*}
	F(t)&\geqslant C_1c_{\gamma} \int_{T_0}^t(1+\eta)^{-\mu}\int_{T_0}^\eta(1+s)^{\mu}\int_{T_0}^s(s-\tau)^{-\gamma}(1+\tau)^{n-1-\frac{n+\mu-1}{2}p}\mathrm{d}\tau\mathrm{d}s\mathrm{d}\eta\\
	&\geqslant \frac{C_1 c_{\gamma}}{n}\int_{T_0}^t(1+\eta)^{-\mu}\int_{T_0}^{\eta}(1+s)^{\mu-\frac{n+\mu-1}{2}p}(s-T_0)^{-\gamma+n}\mathrm{d}s\mathrm{d}\eta\\
	&\geqslant \frac{C_1 c_{\gamma}}{n(n+\mu+1)}\int_{T_0}^t(1+\eta)^{-\mu-\frac{n+\mu-1}{2}p}(\eta-T_0)^{-\gamma+n+\mu+1}\mathrm{d}\eta\\
	&\geqslant \frac{C_1 c_{\gamma}}{n(n+\mu+1)(n+\mu+2)}(1+t)^{-\mu-\frac{n+\mu-1}{2}p}(t-T_0)^{-\gamma+n+\mu+2}
\end{align*}
for any $t\geqslant T_0$. In other words, we have already obtained the first estimate of the functional $F(t)$ from the below by
\begin{align}\label{Eq_Fisrt_F}
	F(t)\geqslant K_0(1+t)^{-\alpha_0}(t-T_0)^{\beta_0}
\end{align}
for any $t\geqslant T_0$, where the multiplicative constant is defined by
\begin{align*}
	K_0:= \frac{C_1 c_{\gamma}}{n(n+\mu+1)(n+\mu+2)}
\end{align*}
and the exponents are given by $\alpha_0:=\mu+(n+\mu-1)p/2$ and $\beta_0:= -\gamma+n+\mu+2$.

Finally, from \eqref{Eq_Iteration_Frame} and \eqref{Eq_Fisrt_F}, we apply generalized Kato's type lemma, i.e. Lemma \ref{Generalized Kato lem}, to get blow-up for the functional $F(t)$ in finite time if
\begin{align*}
	-\frac{(n+\mu)-1}{2}p^2+\left(\frac{(n+\mu)+1}{2}+1-\gamma\right)p+1>0.
\end{align*}
It completes the proof of Theorem \ref{Thm_Blow_Up_mu_0}.

\section*{Acknowledgments}
The authors would like to thank the referee for giving them helpful advice to improve this paper.

\end{document}